\documentclass[12pt, reqno]{amsart}

\usepackage{amsmath, amsthm, amscd, amsfonts, amssymb, graphicx, color}
\usepackage[bookmarksnumbered, colorlinks, plainpages]{hyperref}
\usepackage{dsfont}
\input{mathrsfs.sty}
\usepackage{subfigure}
 \makeatletter
\let\reftagform@=\tagform@
\def\tagform@#1{\maketag@@@{(\ignorespaces\textcolor{blue}{#1}\unskip\@@italiccorr)}}
\renewcommand{\eqref}[1]{\textup{\reftagform@{\ref{#1}}}}
\makeatother
\usepackage{hyperref}
\hypersetup{colorlinks=true, linkcolor=red, anchorcolor=green,
citecolor=cyan, urlcolor=red, filecolor=magenta, pdftoolbar=true}

\newtheorem{theorem}{Theorem}[section]
\newtheorem{lemma}[theorem]{Lemma}
\newtheorem{proposition}[theorem]{Proposition}
\newtheorem{corollary}[theorem]{Corollary}
\theoremstyle{definition}

\newtheorem{example}[theorem]{Example}

\theoremstyle{remark}
\newtheorem{remark}[theorem]{Remark}
\numberwithin{equation}{section}
  \def\etal{et al.\,}
\begin{document}
\setcounter{page}{1}

\title[Superquadratic functions]{Improved  eigenvalue inequalities via two major subclasses of  Superquadratic functions   }

\author[M. Kian ]{Mohsen Kian }

\address{Mohsen Kian: Department of Mathematics, University of Bojnord, P.O. Box
1339, Bojnord 94531, Iran}
\email{kian@ub.ac.ir }

\subjclass[2020]{Primary:   26A51, 15A45  Secondary:  15A18, 47A56.}

\keywords{Superquadratic function, convex function,  eigenvalue, positive definite matrix}

\begin{abstract}
  There exist two major subclasses in the class of  superquadratic functions, one  comprises concave and decreasing functions, while the other consists of convex and monotone increasing functions.  Leveraging this distinction, we introduce eigenvalue inequalities for each case. The characteristics of these functions allow us to advance our findings in two ways: firstly, by refining existing results related to eigenvalues for convex functions, and secondly, by deriving complementary inequalities for other function types. To bolster our claims, we will provide illustrative examples.

\end{abstract}

\maketitle

\section{Introduction}
Convexity is one of the most basic notions in both theoretic and applied mathematics. Convex functions play central role in many fields.  Geometrically, a convex function lies above its tangent lines.  Equivalently, if  a function $f : J\subseteq\mathds{R}\to \mathds{R}$ is convex, then it has a non-empty subdifferential at any interior point of its domain, say, for every $s\in J$,
  there exists a $a
\in \mathds{R}$ such that
\begin{align}
f\left( t \right) \ge f\left( s \right) + a \left( {t - s}
\right) \label{eq1.2}
\end{align}
for all $t\in J$. According to this,  Abramovich \etal \cite{ajs} introduced the class of superquadratic functions. A function $f$ in this class,   requires to be   above its tangent line plus a translation of $f$
itself. Symbolically, a function $f:[0,\infty)\to\mathds{R}$  is called superquadratic provided that for all $t\geq0$ there exists a constant $C_t\in\mathbb{R}$  such that
  \begin{eqnarray*}
    f(s)\geq f(t)+C_t(s-t)+f(|s-t|)
  \end{eqnarray*}
for all $s\geq0$.   If $-f$ is a superquadratic function, then  $f$ is  called \textit{subquadratic}.  Every superquadratic function $f:[0,\infty)\to\mathds{R}$ enjoys  a Jensen inequality as
\begin{align}\label{spq-2}
  f(\alpha t+(1-\alpha)s)\leq \alpha f(t) +(1-\alpha)f(s)-\alpha f((1-\alpha)|t-s|)-(1-\alpha)f(\alpha |t-s|)
\end{align}
for all $t,s\geq0$ and $\alpha\in[0,1]$. In a first look at this inequality, they seems to be stronger than convex functions. However, this is not the case in general. Basic properties of superquadratic functions   read  as follows.
 \begin{lemma}\label{spq-1}\cite{ajs,abm}
 If   $f:[0,\infty)\to\mathds{R}$  is   superquadratic, then
 \begin{align*}
   &{\rm (i)}\ \ f(0)\leq0;\\
    &{\rm (ii)}\  \mbox{If $f(0)=f'(0)=0$ and  $f$ is differentiable at $t$, then $C_t=f'(t)$};\\
 &{\rm (iii)}\ \mbox{ If $f\geq0$, then $f$ is convex and $f(0)=f'(0)=0$}.
 \end{align*}
  \end{lemma}

Basic properties of this class of functions  can be found in \cite{ajs}. Extension   to several variable functions have been given in \cite{abm}. Some classical inequalities  for convex functions like Jensen-Stefensen inequality,  Hermite-Hadamard inequality and Fejer inequality,   were proved for this class in \cite{Ab-Ba-Pe,SIP}.  Regarding operator extensions of the classical results, we proved a   variant of the Jensen inequality for superquadratic functions involving operators in \cite{K}. The authors of  \cite{Ba-Ma-Pe} presented some Mercer type operator inequalities for superquadratic functions. Regarding Mercer inequality, E. Anjdani \cite{An} gave another variant and also presented some reverse results in \cite{An-Ch}.        In \cite{Ki-Al}, some trace inequalities have been shown for these functions. Recently, the authors of \cite{Kr-Mo-Sa} have studied  logarithmically  superquadratic functions and their properties.

Looking  superquadratic functions, one can notice that there exist two main subclasses in the class of  superquadratic functions, one contains concave and decreasing functions and the other, contains convex and monotone increasing functions. According to this point of view,  we   use the nature of these functions to present eigenvalue inequalities  and then apply our results in two directions. First improving some known results concerning eigenvalues for convex functions and second, obtaining some complimentary  inequalities  for some  other type functions. Some examples will clarify our results.

We need some basic facts from matrix analysis.

 \subsection{Preliminaries}
Throughout the paper,  we assume that $\mathbb{M}_n$ is the   algebra of $n\times n$ matrices with complex entries. A Hermitian matrix $A\in\mathbb{M}_n$ is called positive semi-definite and denoted by $A\geq0$ (positive definite and denoted by $A>0$) if all of its eigenvalues are non-negative (positive). This establishes a partial order on the set of Hermtian matrices, the celebrated L\"{o}wner partial order: $A\leq B$ if and only if $B-A\geq 0$. We denote the set of positive semi-definite matrices and positive definite matrices by $\mathbb{M}_n^+$ and $\mathbb{M}_n^{++}$, respectively. For a Hermitian matrix $A$, we mean by $\lambda^\downarrow(H)$, the vector of eigenvalues of $H$ arranged in decreasing order and with counted their multiplicities, i.e., $\lambda^\downarrow(H)=(\lambda_1^\downarrow(H),\dots,\lambda_n^\downarrow(H))$.  The minimax principle regarding eigenvalues of a Hermitian matrix reads as follows.

 \begin{lemma}[The Minimax Principle]\cite[Corollary III.1.2]{bh}\label{minimax}
If $H\in\mathbb{M}_n$ is a Hermitian matrix, then
\begin{align}
  \lambda_k^\downarrow(H)=\max_{\substack{\mathcal{M}\subseteq\mathds{C}^n\\ \mathrm{dim}(\mathcal{M})=k}}\min_{\substack{x\in\mathcal{M}\\ \|x\|=1}}\langle Hx,x\rangle
  =\min_{\substack{\mathcal{M}\subseteq\mathds{C}^n\\ \mathrm{dim}(\mathcal{M})=n-k+1}}\max_{\substack{x\in\mathcal{M}\\ \|x\|=1}}\langle Hx,x\rangle.
\end{align}
\end{lemma}

For two Hermitian matrices $H$ and $K$, the inequality $\lambda^\downarrow(H)\leq\lambda^\downarrow(K)$ means $\lambda_j(H)\leq \lambda_j(K)$ for all $j=1,\dots,n$. It is easy to see that
\begin{align}\label{unit}
\lambda^\downarrow(H)\leq \lambda^\downarrow(K)\qquad \Longleftrightarrow\qquad  H\leq U^*KU\quad\mbox{for some unitary $U\in\mathbb{M}_n$}.
\end{align}

It is known that if $f$ is a convex function on an interval $J$, then
\begin{align}\label{mp}
 f(\langle Ax,x\rangle)\leq \langle f(A)x,x\rangle
\end{align}
  for every  Hermitian matrix $A$ with eigenvalues in $J$ and any unit vector $x\in\mathds{C}^n$. If $f$ is concave, the inequality is reversed.   We showed in \cite{K,KS} that if $f$ is a superquadratic function, then
\begin{align}\label{k}
f(\langle Ax,x\rangle)\leq \langle f(A)x,x\rangle-\langle f(|A-\langle Ax,x\rangle|)x,x\rangle
\end{align}
holds for any unit vector $x$ and any positive semi-definite matrix $A$. If in  addition $\Phi$ is a unital positive linear map on $\mathbb{M}_n$, then
\begin{align}\label{kd}
f(\langle \Phi(A)x,x\rangle)\leq \langle \Phi(f(A))x,x\rangle-\langle \Phi(f(|A-\langle \Phi(A)x,x\rangle|))x,x\rangle.
\end{align}

It has been shown in \cite{Au-Si} that if $f:J\to\mathds{R}$ is convex and increasing, then
\begin{align}\label{original}
 \lambda^\downarrow(f(1-\alpha)A+\alpha B) \leq \lambda^\downarrow((1-\alpha)f(A) +\alpha   f(B))
\end{align}
for  and all Hermitian matrices $A$ and $B$, whose eigenvalues are contained in $J$. In particular,
\begin{align}\label{aj-p}
 \lambda^\downarrow\left(\left(\frac{A+B}{2}\right)^p\right)  \leq \lambda^\downarrow\left(\frac{A^p+B^p}{2}\right)\qquad (p\geq1).
\end{align}
In the case where the convexity of $f$ is changed or replaced by other properties, see \cite{Au-Bo,ks,L-P}.

%%%%%%%%%%%%%%%%%%%%%%%%%%%%%%%%%%%%%%%%%%%%%%%%%%%%%%%%%%%%%%%%%%%%%%%%%%%%%%%%%%%%%%%%%%%%%%%
\section{Main Result}
Our motivation is based on this fact that the class of superquadratic functions contains two major subclasses: first, those functions which are convex and increasing, like $f(x)=x^p$ for $p\geq2$ and $f(x)=x^2\log(x)$, and second, those functions which are concave and decreasing, like $f(x)=-x^q$ for $1\leq q\leq 2$ and $f(x)=-\left(1+x^{1/r}\right)^r$ for $r\in(0,1]$. This helps us to proceed in two directions, one for improving some results  concerning  convex functions, and two, establishing reverse inequalities for some other   functions.

Graphs of some typical elements in the two subclasses of  superquadratic functions have been given in Figure~1.

		\begin{center}
		\begin{figure}[h!]%\label{papa}
\caption{Graphs of some functions in the two subclasses of  superquadratic functions}
			\subfigure[Increasing convex functions]{\includegraphics[height=4cm,width=6cm]{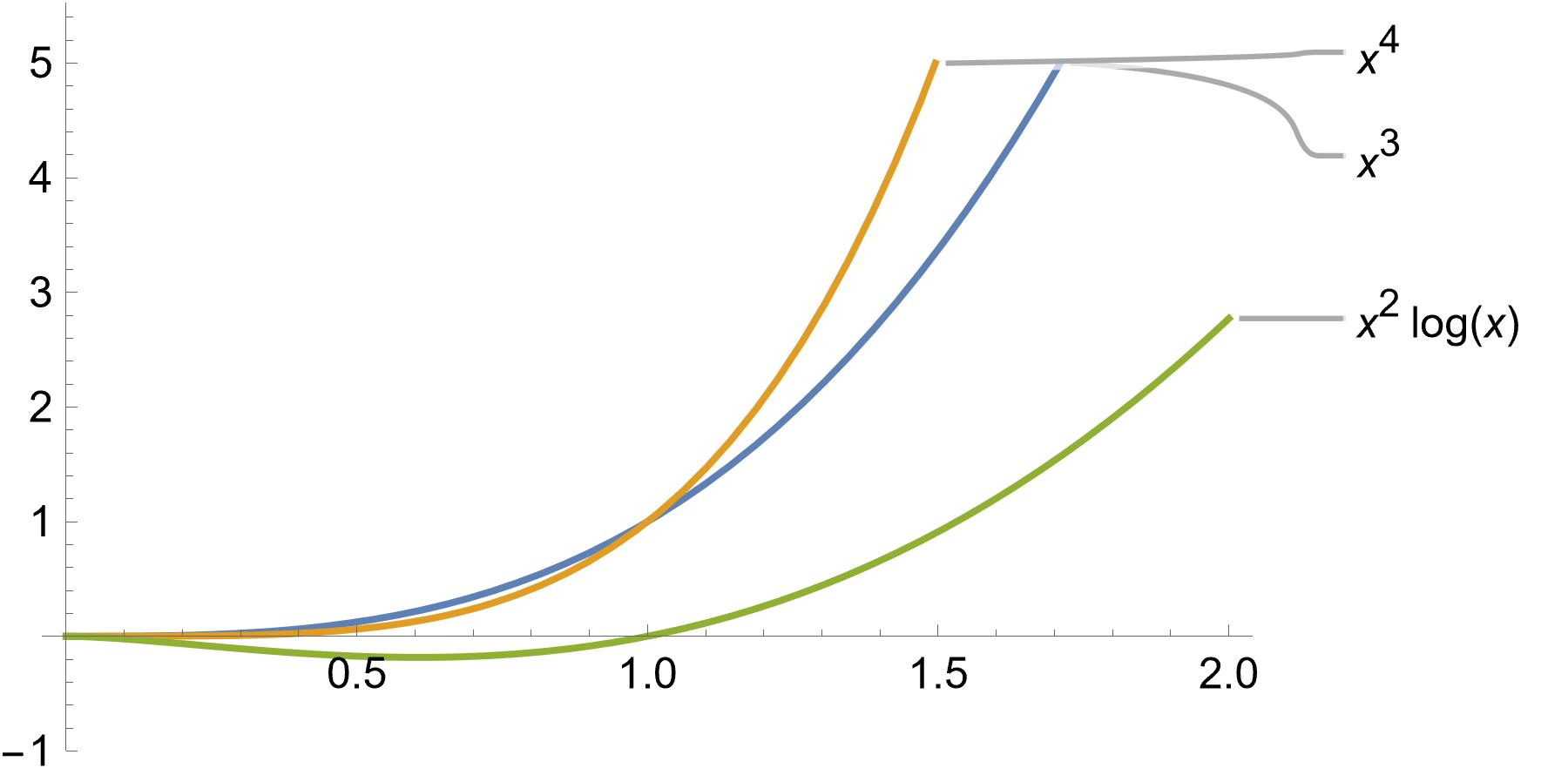}\label{i-c}}
			\hspace{1cm}
			\subfigure[Decreasing concave functions]{\includegraphics[height=4cm,width=6cm]{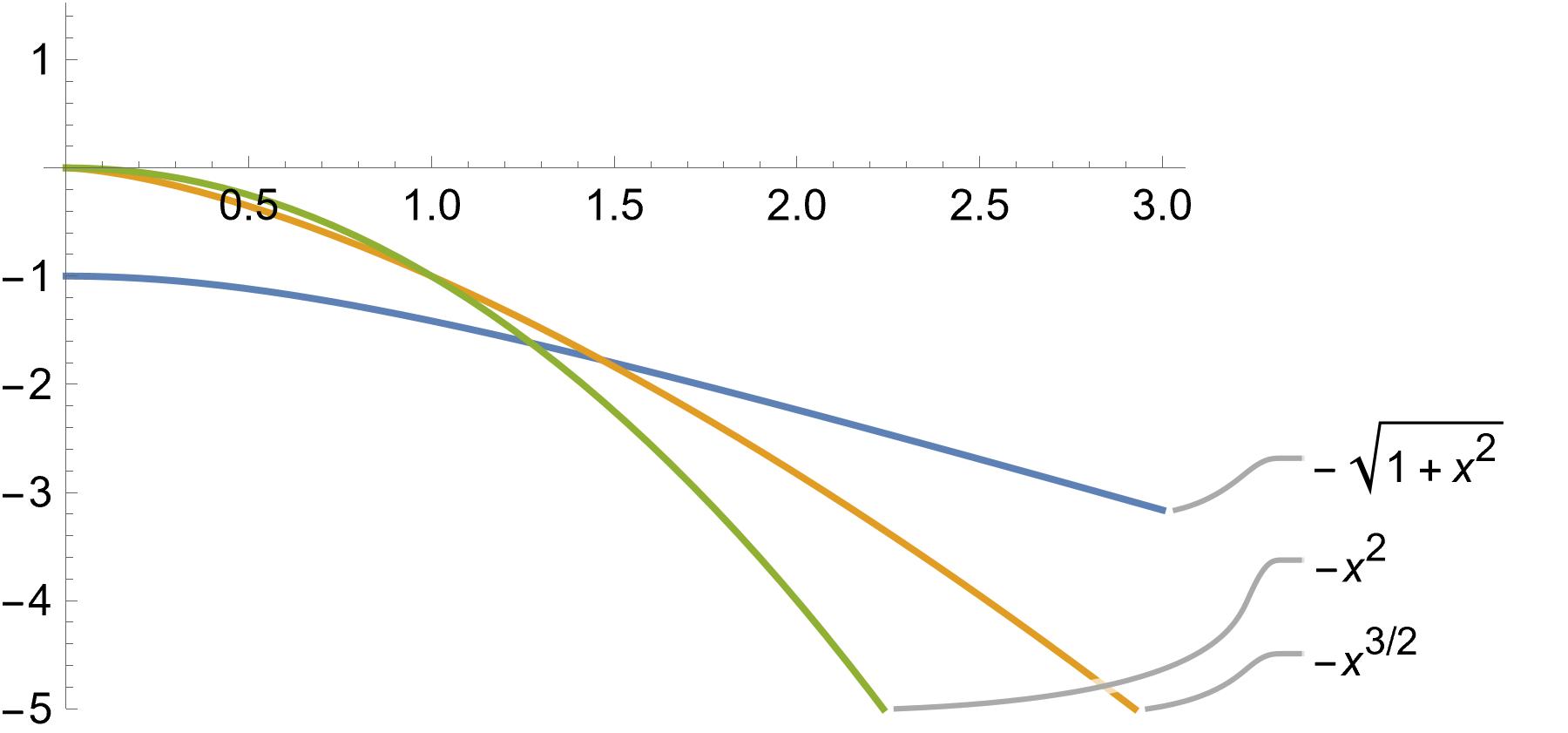}\label{d-c}}
					\end{figure}
	\end{center}

\bigskip

In our first result, we consider those functions which are concave and decreasing.

%%--------------------------------------------------

\begin{theorem}\label{th-n-1}
  Let $f:[0,\infty)\to\mathds{R}$ be a superquadratic function, $\alpha\in[0,1]$ and let $A,B\in~\mathbb{M}_n^+$. If $f$ is concave and decreasing, then
\begin{align}\label{thq-2}
\lambda_k^\downarrow(f(1-\alpha)A+\alpha B)\leq  \lambda_k^\downarrow(S)\qquad (k=1,\dots,n),
\end{align}
where
 \begin{align*}
 S&=(1-\alpha)(  f(A) -f(\lambda_1(A)-\lambda_n(A))- f(\alpha|A-B|))  \\
  & \qquad\qquad\qquad\qquad +\alpha (   f(B)  -f(\lambda_1(B)-\lambda_n(B))  -  f((1-\alpha)|A-B|) ).
\end{align*}
\end{theorem}
\begin{proof}
 For a Hermitian matrix $H$ and  a fixed index $k\in\{1,\dots,n\}$, by the minimax principle, Lemma \ref{minimax},  there exists a $k$-dimensional vector subspace $\mathcal{M}\subseteq\mathds{C}^n$ such that
$$\lambda_{n-k+1}^\downarrow(H)=\max\{\langle Hx,x\rangle;\, x\in\mathcal{M},\,\|x\|=1\}.$$
Since $f$ is decreasing, we have
\begin{align}\label{new11}
\begin{split}
 \lambda_k^\downarrow\left[f\left((1-\alpha)A+\alpha B\right)\right]
 & =f\left(\lambda_{n-k+1}^\downarrow\left((1-\alpha)A+\alpha B\right)\right) \\
&= f\left(\max\{\langle ((1-\alpha)A+\alpha B)x,x\rangle;\, x\in\mathcal{M},\,\|x\|=1\}\right)\\
&=\min\{f(\langle ((1-\alpha)A+\alpha B)x,x\rangle);\, x\in\mathcal{M},\,\|x\|=1\}.
\end{split}
\end{align}
Since $f$ is superquadratic, for every unit vector  $x\in\mathds{C}^n$ \eqref{spq-2} entails
\begin{align}\label{new22}
 f(\langle  ((1-\alpha)A+\alpha B )x,x\rangle)&=
 f((1-\alpha)\langle Ax,x\rangle+\alpha \langle B x,x\rangle)\nonumber\\
   &\leq  (1-\alpha)f(\langle Ax,x\rangle)+\alpha f(\langle B x,x\rangle)\\
  &\quad - (1-\alpha)f(\alpha|\langle(A-B)x,x\rangle|)+ \alpha f((1-\alpha)|\langle(A-B)x,x\rangle|).\nonumber
\end{align}
 In addition, applying \eqref{mp} for the  convex   function $t\mapsto|t|$,  implies that
  $$|\langle(A-B)x,x\rangle|\leq  \langle|A-B|x,x\rangle$$
 holds  for every unit vector $x\in\mathds{C}^n$. Since $f$ is concave and monotone decreasing,  this concludes
  \begin{align}\label{eb2}
    \begin{split}
    f(c|\langle(A-B)x,x\rangle|)&\geq  f(c\langle|A-B|x,x\rangle)
       \geq\left\langle  f(c|A-B|)x,x\right\rangle
    \end{split}
  \end{align}
  for every positive real number $c\in\mathds{R}^+$, where in the second inequality we applied  \eqref{mp}  for the concave function $f$.
Moreover, it is clear that   if $Y$ is a positive  semi-definite matrix, then $\lambda_n(Y)\leq \langle Yx,x\rangle \leq \lambda_1(Y)$  holds
for every unit vector $x\in\mathds{C}^n$,  whence $|Y- \langle Yx,x\rangle|\leq \lambda_1(Y)-\lambda_n(Y)$. Consequently,
\begin{align}\label{np}
  \langle f(|Y-\langle Yx,x\rangle|)x,x\rangle \geq f(\lambda_1(Y)-\lambda_n(Y)),
\end{align}
as  $f$ is decreasing.  Accordingly,   utilising \eqref{k}  implies that
  \begin{align}\label{we1}
  \begin{split}
    f(\langle Yx,x\rangle) &\leq  \langle f(Y)x,x\rangle -  \langle f(|Y-\langle Yx,x\rangle|)x,x\rangle\\
 &\leq \langle f(Y)x,x\rangle -f(\lambda_1(Y)-\lambda_n(Y))\qquad(\mbox{by \eqref{np}})
  \end{split}
  \end{align}
  holds for every positive semi-definite  matrix $Y$.
It follows from \eqref{eb2} and \eqref{we1} that
{\small\begin{align}\label{we3}
  \begin{split}
&(1-\alpha)f(\langle Ax,x\rangle)  +\alpha f(\langle B x,x\rangle)\\
&\quad-(1-\alpha)f(\alpha|\langle(A-B)x,x\rangle|) - \alpha f((1-\alpha)|\langle(A-B)x,x\rangle|))\\
&\leq (1-\alpha)(\langle f(A)x,x\rangle -f(\lambda_1(A)-\lambda_n(A)))   +\alpha ( \langle f(B) x,x\rangle
-f(\lambda_1(B)-\lambda_n(B)) )\\
&\quad-(1-\alpha)\left\langle  f(\alpha|A-B|)x,x\right\rangle - \alpha \left\langle f((1-\alpha)|A-B|)x,x\right\rangle.
\end{split}
\end{align}}
Combining   \eqref{new11},\eqref{new22} and \eqref{we3} we obtain
\begin{align*}
\lambda_k^\downarrow(f(1-\alpha)A+\alpha B)& \leq
\min\{\langle Sx,x\rangle ,\, x\in\mathcal{M},\,\|x\|=1\}\\
&\leq \lambda_k^\downarrow(S),
\end{align*}
  where the second inequality is a  consequence of  the minimax principle and
\begin{align*}
 S&=(1-\alpha)(  f(A) -f(\lambda_1(A)-\lambda_n(A))- f(\alpha|A-B|))  \\
  & \qquad\qquad\qquad\qquad +\alpha (   f(B)  -f(\lambda_1(B)-\lambda_n(B))  -  f((1-\alpha)|A-B|) ).
\end{align*}
This completes the proof of \eqref{thq-2}.
\end{proof}

%%---------------------------------------------
\begin{corollary}\label{co-n-1}
 Let $f:[0,\infty)\to\mathds{R}$ be a superquadratic function  and let $A,B\in~\mathbb{M}_n^+$. If $f$ is concave and decreasing, then
   {\small\begin{align}\label{co1-q4}
 f\left(\frac{A+B}{2}\right)
&\leq   V^*\left[\frac{f(A)+f(B)}{2}- f\left(\left|\frac{A-B}{2}\right|\right)
-\frac{f(\lambda_1(A)-\lambda_n(A))+f(\lambda_1(B)-\lambda_n(B))}{2}\right]V
\end{align}}
for some unitary matrix $V\in\mathbb{M}_n$.
\end{corollary}
\begin{proof}
This is a consequence of applying  Theorem~\ref{th-n-1} when $\alpha=1/2$ and \eqref{unit}.
\end{proof}
%%-----------------------------------------------------------------
The function $f(t)=-t^q$ is a superquadratic function for $q\in[1,2]$. It is concave and decreasing and satisfies the condition of Corollary~\ref{co-n-1}.   The next corollary then follows. In-fact, it provides a reverse inequality for \eqref{aj-p}.

\begin{corollary}
Let $A,B$ be positive semi-definite matrices. If $1\leq q\leq 2$, then the exists a unitary $V$ such that
 {\small\begin{align}\label{co1-q4-1}
\left(\frac{A^q+B^q}{2}\right)\leq V\left(\frac{A+B}{2}\right)^qV^*+\left|\frac{A-B}{2}\right|^q
+\frac{(\lambda_1(A)-\lambda_n(A))^q+ (\lambda_1(B)-\lambda_n(B))^q}{2}I.
\end{align}}
\end{corollary}

Regarding estimation of eigenvalues, we can state the following result.
\begin{corollary}\label{poer1}
 Let $A,B$ be positive semi-definite matrices. If $1\leq q\leq 2$, then
   {\small\begin{align}\label{co1-q4}
\lambda^\downarrow\left[(A+B)^q\right]
&\geq    \lambda^\downarrow\left[2^{q-1}\left(A^q+B^q-(\lambda_1(A)-\lambda_n(A))^q-(\lambda_1(B)-\lambda_n(B))^q\right)
-|A-B|^q \right].
\end{align}}
\end{corollary}

Theorem \ref{th-n-1} provide a converse to \eqref{original}, see \cite[Theorem 2.9]{Au-Si}. Also, Corollary~\ref{poer1} gives a converse to the eigenvalue inequality \eqref{aj-p}, presented by Aujla and Silva.
We give a toy  example and compare the result with the original   estimation. Assume that $q=3/2$ and consider the positive semi-definite matrices
\begin{align*}
  A=\begin{pmatrix}
      5 & -1 \\
      -1 & 5
    \end{pmatrix} \quad\mbox{and}\quad  B=\begin{pmatrix}
      2 & 0 \\
      0 & 4
    \end{pmatrix}.
\end{align*}
Thanks to Matlab software we compute
\begin{align*}
 \lambda^{\downarrow}\left(\frac{A^q+B^q}{2}\right)=(6.266,10.4967)\quad\mbox{and}\quad \lambda^{\downarrow}\left(\left(\frac{A+B}{2}\right)^q\right)=(5.9754,10.2125)
\end{align*}
and
\begin{align*}
 \lambda^{\downarrow}\left(\frac{A^q+B^q}{2}-\left|\frac{A-B}{2}\right|^q
-\frac{(\lambda_1(A)-\lambda_n(A))^q+ (\lambda_1(B)-\lambda_n(B))^q}{2}I \right)=(2.1248,6.5921)
\end{align*}
and so
{\small\begin{align*}
  \lambda^{\downarrow}\left(\frac{A^q+B^q}{2}-\left|\frac{A-B}{2}\right|^q
-\frac{(\lambda_1(A)-\lambda_n(A))^q+ (\lambda_1(B)-\lambda_n(B))^q}{2}I \right)&\leq \lambda^{\downarrow}\left(\left(\frac{A+B}{2}\right)^q\right)\\
&\leq \lambda^{\downarrow}\left(\frac{A^q+B^q}{2}\right).
\end{align*}}
%%-------------------------------

\bigskip

We present another estimation for eigenvalues of concave and decreasing superquadratic functions, using \eqref{kd}.
%%--------------------------------------------------------
\begin{theorem}\label{es-th2}
Let $f:[0,\infty)\to\mathds{R}$ be a superquadratic function and let $\Phi:\mathbb{M}_n\to\mathbb{M}_m$ be a unital positive linear map.      If  $f$ is concave and decreasing, then
{\small\begin{align}\label{co2-q2}
\lambda_k^\downarrow\left(f\left(\Phi(A)\right)\right) \leq \lambda_k^\downarrow(\Phi(f(A))) - f(\lambda_1(A)- \lambda_n(A)) ,\qquad (k=1,\dots,n)
\end{align}}
 for every $A\in\mathbb{M}_n^+$.
\end{theorem}
\begin{proof}
Let $x\in\mathds{C}^n$ be a unit vector. Since $\Phi$ is  unital, it is evident that $\lambda_n(A)\leq \Phi(A)\leq \lambda_1(A)$. Hence $|A-\langle \Phi(A) x,x\rangle  |\leq \lambda_1(A)- \lambda_n(A)$, as the matrices $A$ and $\langle \Phi(A) x,x\rangle I_n$  commute. Since $f$ is decreasing, it follows that
$f(|A-\langle \Phi(A) x,x\rangle|)\geq f(\lambda_1(A)- \lambda_n(A))$, whence
\begin{align}\label{2a}
\langle\Phi(f(|A-\langle \Phi(A) x,x\rangle|)) x,x\rangle\geq f(\lambda_1(A)- \lambda_n(A)).
\end{align}

Since  $f$ is superquadratic, we obtain from \eqref{kd} and \eqref{2a} that
\begin{align}\label{1a}
f( \langle \Phi(A) x,x\rangle)
   \leq  \langle \Phi(f(A)) x,x\rangle- f(\lambda_1(A)- \lambda_n(A)).
\end{align}
  Fix $k\in\{1,\dots,n\}$. The minimax principle ensures that  there exists a $k$-dimensional vector subspace $\mathcal{M}$ of $\mathds{C}^n$ such that
$$\lambda_{n-k+1}^\downarrow(\Phi(A))=\max\{\left\langle  \Phi(A) x,x\right\rangle;\, x\in\mathcal{M},\,\|x\|=1\}.$$
Furthermore, since $f$ is decreasing,  we have $\lambda_k^\downarrow(f(\Phi(A)))=f(\lambda_{n-k+1}^\downarrow(\Phi(A)))$. Therefore,
\begin{align}\label{3a}
\begin{split}
\lambda_k^\downarrow\left(f\left(\Phi(A)\right)\right)&=f\left(\max\{\left\langle  \Phi(A) x,x\right\rangle,\, x\in\mathcal{M},\,\|x\|=1\}\right)\\
&=\min\{f(\left\langle \Phi(A)  x,x\right\rangle),\, x\in\mathcal{M},\,\|x\|=1\},
\end{split}\end{align}
where we use the fact that $f$ is decreasing in the second equality. Hence
\begin{align*}
  \lambda_k^\downarrow\left(f\left(\Phi(A)\right)\right)&\leq \min\{\langle (\Phi(f(A)) - f(\lambda_1(A)- \lambda_n(A)))x,x\rangle,\, x\in\mathcal{M},\,\|x\|=1\}\\
  &\qquad\qquad\qquad\qquad(\mbox{by \eqref{1a} and \eqref{3a}})\\
  &\leq \lambda_k^\downarrow\left(\Phi(f(A)) - f(\lambda_1(A)- \lambda_n(A))\right),
\end{align*}
by using the minimax principle again.
 \end{proof}

\begin{remark}
  As a consequence of Theorems~\ref{es-th2}, we  have the trace inequality
  $$\mathrm{tr}(f(\Phi(A)))\leq \mathrm{tr}(\Phi(f(A))-  f(\lambda_1(A)- \lambda_n(A)))$$
  under the same hypothesis.
\end{remark}
%----------------------------------------
\begin{remark}
 We compare the two estimations in Theorems~\ref{es-th2} and \ref{th-n-1}.   Suppose that the linear mapping $\Phi:\mathbb{M}_{2n}\to\mathbb{M}_n$ is defined by
 \begin{align}\label{phi}
 \Phi\left(\begin{pmatrix}
                X_{11} & X_{12} \\
                X_{21} & X_{22}
              \end{pmatrix}\right)
              =\begin{pmatrix}
                C^* & D^*
              \end{pmatrix}
              \begin{pmatrix}
                X_{11} & X_{12} \\
                X_{21} & X_{22}
              \end{pmatrix}
              \begin{pmatrix}
                C \\
                D
              \end{pmatrix}
 \end{align}
 for some $C,D\in\mathbb{M}_n$. Then clearly $\Phi$ is positive. If $C^*C+D^*D=I$, then $\Phi$ is unital, too.
 Let $A,B\in\mathbb{M}_n^+$ and put $X=A\oplus B$. Then $\lambda_1(X)=\max\{\lambda_1(A),\lambda_1(B)\}$ and  $\lambda_n(X)=\min\{\lambda_n(A),\lambda_n(B)\}$. Applying \eqref{co2-q2}  yields
\begin{align}\label{th2-qw22}
\begin{split}
  \lambda^\downarrow\left(f(C^*AC+D^*BD)\right)&\leq\lambda^\downarrow( C^*f(A)C+D^*f(B)D)\\
   &\quad- f\left(\max\{\lambda_1(A),\lambda_1(B)\}-\min\{\lambda_n(A),\lambda_n(B)\}\right)).
   \end{split}
\end{align}
  In particular, if $\alpha\in[0,1]$, then  with  $C=\alpha I_n$ and $D=(1-\alpha)I_n$  this turns into
\begin{align}\label{th2-qw2}
\begin{split}
  \lambda^\downarrow(f((1-\alpha)A+\alpha B))&\leq\lambda^\downarrow( (1-\alpha)f(A)+\alpha f(B))\\
   &\quad- f\left(\max\{\lambda_1(A),\lambda_1(B)\}-\min\{\lambda_n(A),\lambda_n(B)\}\right)).
   \end{split}
\end{align}

\end{remark}

Inequality \eqref{th2-qw2} gives a reverse inequality to \cite[Theorem 2.9]{Au-Si} for the subclass of superquadratic functions, which are concave and decreasing. Let us give some  examples.

\begin{example}
  Assume that $a_1\geq a_2\geq 0$ and $b_1\geq b_2\geq0$ are non-negative real numbers and consider positive semi-definite matrices $A=\mathrm{diag}(a_1,a_2)$ and $B=\mathrm{diag}(b_1,b_2)$. According to   \eqref{th2-qw2} we have
  $$f\left(\frac{a_1+b_1}{2}\right)\leq \frac{f(a_1)+f(b_1)}{2}-f\left(\max\{a_1,b_1\}-\min\{a_2,b_2\}\right)$$
  and
  $$f\left(\frac{a_2+b_2}{2}\right)\leq \frac{f(a_2)+f(b_2)}{2}-f\left(\max\{a_1,b_1\}-\min\{a_2,b_2\}\right)$$
 which is trivially a converse of the Jensen inequality for the concave function $f$.
\end{example}

\begin{example}
Inequality \eqref{th2-qw2} gives a reverse of \eqref{original} presented in \cite[Theorem 2.9]{Au-Si}. Let $\alpha=1/2$, $q=4/3$, $f(x)=-x^q$  and put
  \begin{align*}
  A=\begin{pmatrix}
      3   &  1  &   0\\
     1     &4 &    0\\
     0     &0&     3
    \end{pmatrix} \quad\mbox{and}\quad  B=\begin{pmatrix}
      5 &     0  &    0\\
     0   &   4   &   1\\
     0   &   1  &    4
    \end{pmatrix}.
\end{align*}
Then
 {\small \begin{align*}
\lambda^\downarrow\left(\frac{A+B}{2}\right)^q=(7.7024,5.8831, 4.5182),\quad \lambda^\downarrow\left(\frac{A^q+B^q}{2}\right)
=(7.7327,6.0201,4.5596).
\end{align*}}
Also
{\small \begin{align*}
\left(\max\{\lambda_1(A),\lambda_1(B)\}-\min\{\lambda_n(A),\lambda_n(B)\}\right)^q\simeq 3.608
\end{align*}}
and we have
{\small \begin{align*}
\lambda^\downarrow\left(\frac{A+B}{2}\right)^q\geq \lambda^\downarrow\left(\frac{A^q+B^q}{2}\right)-\left(\max\{\lambda_1(A),\lambda_1(B)\}-\min\{\lambda_n(A),\lambda_n(B)\}\right)^q.
\end{align*}}

\end{example}
\begin{remark}

  The two estimations  \eqref{thq-2} and \eqref{th2-qw2} can be more precise than each other occasionally.     Suppose that $q=4/3$ and consider the superquadratic function $f(x)=-x^q$. With $\alpha=1/2$ and  the two positive matrices
  \begin{align*}
  A=\begin{pmatrix}
      5 & -1 \\
      -1 & 5
    \end{pmatrix} \quad\mbox{and}\quad  B=\begin{pmatrix}
      4 & 1 \\
      1 & 5
    \end{pmatrix}
\end{align*}
calculating sing Matlab software gives
  \begin{align*}
  \lambda^{\downarrow}\left(\frac{A^q+B^q}{2}
  -\left(\max\{\lambda_1(A),\lambda_1(B)\}-\min\{\lambda_n(A),\lambda_n(B)\}\right)^q\right)
  =(3.9202,    5.0212)
\end{align*}
and
\begin{align*}
 \lambda^{\downarrow}\left(\frac{A^q+B^q}{2}-\left|\frac{A-B}{2}\right|^q
-\frac{(\lambda_1(A)-\lambda_n(A))^q+ (\lambda_1(B)-\lambda_n(B))^q}{2}I \right)
=(3.6099,  4.9944),
\end{align*}
whence the estimation derived by Theorem~\ref{es-th2} is more accurate. Now, consider the two matrices
 \begin{align*}
  A=\begin{pmatrix}
      9 & -1 \\
      -1 & 8
    \end{pmatrix} \quad\mbox{and}\quad  B=\begin{pmatrix}
      5 & 1 \\
      1 & 5
    \end{pmatrix}.
\end{align*}
We have
  \begin{align*}
  \lambda^{\downarrow}\left(\frac{A^q+B^q}{2}
  -\left(\max\{\lambda_1(A),\lambda_1(B)\}-\min\{\lambda_n(A),\lambda_n(B)\}\right)^q\right)
  =(2.3178,3.7477)
\end{align*}
and
\begin{align*}
 \lambda^{\downarrow}\left(\frac{A^q+B^q}{2}-\left|\frac{A-B}{2}\right|^q
-\frac{(\lambda_1(A)-\lambda_n(A))^q+ (\lambda_1(B)-\lambda_n(B))^q}{2}I \right)
= (6.6286,9.4128).
\end{align*}
Accordingly, we face a more accurate estimation in Theorem~\ref{th-n-1} this time.
\end{remark}

%%---------------------------------------------------------------------------

%%--------------------------------------------------------------------------

\bigskip

Another major class of superquadratic functions comprises  convex and increasing functions. In-fact, Lemma~\ref{spq-1} ensures that every   positive superquadratic function falls into this class.   In the next result, we consider this class of functions. Before presenting the main result of this part, we need to introduce some notations.

Let $m<M$ be  real numbers. For a real function $g$ defined on an interval $[m,M]$, we fix the notations
{\small\begin{align*}
\mu_g(m,M)=\frac{g(M)-g(m)}{M-m}\quad\mbox{and}\quad \nu_g(m,M)=\frac{ Mg(m)-mg(M)}{M-m}.
\end{align*}}

We recall a counterpart to \eqref{mp} as follows.
\begin{lemma}\label{lm-fmps}\cite[Corollary 2.28]{FMPS}
Let $Z\in\mathbb{M}_n$ be a Hermitian matrix,     whose eigenvalues are contained in the real interval $[m,M]$ and let $x\in\mathds{C}^n$ be a unit vector. If $g$ is a  strictly convex twice differentiable function on $[m,M]$, then
\begin{align}\label{jen-rev}
  \langle g(Z) x,x\rangle \leq \gamma(m,M,g)\, g(\langle Zx,x\rangle)
\end{align}
in which $\gamma(m,M,g)=g(t_0)^{-1}(\mu_g t_0+\nu_g)$ and $t_0$ is the unique solution of the equation  $\mu_g g(t)=g'(t)(\mu_g t+\nu_g)$.
\end{lemma}
In particular, if $Z$ is a positive definite matrix and  $p\in\mathds{R}\backslash[0,1]$, with the function  $g(t)=t^p$  defined on $(0,\infty)$,  inequality \eqref{jen-rev} turns into
\begin{align}\label{jen-rev-pow}
  \langle Z^p x,x\rangle \leq K(m,M,p)  \langle Zx,x\rangle^p,
\end{align}
in which $K(m,M,p)$ is the so-called generalized Kantorovich constant. If $p\in[0,1]$, a reverse inequality holds in \eqref{jen-rev-pow}.

\bigskip

Now we give our next result.

%----------------------------------------------------------------
\begin{theorem}\label{th-t1}
  Let $f:(0,\infty)\to(0,\infty)$ be a twice differentiable superquadratic function, $\alpha\in[0,1]$ and let $A,B\in\mathbb{M}_n^{++}$. If  $f$  is  strictly convex   and $A-B$ is invertible, then
\begin{align}\label{thq-1}
\lambda_k^\downarrow\left(f\left((1-\alpha)A+\alpha B\right)\right)&\leq \lambda_k^\downarrow(T) \qquad (k=1,\dots,n),
\end{align}
where
{\small\begin{align*}
T &=(1-\alpha)f(A)+ \alpha f(B)\\
&\,\,- (1-\alpha) \gamma(\alpha\lambda_n, \alpha\lambda_1,g)^{-1}  f(\alpha|A-B|) - \alpha\gamma((1-\alpha)\lambda_n,(1-\alpha)\lambda_1,g)^{-1}  f((1-\alpha)|A-B|)
\end{align*}}
and $\lambda_n$ and $\lambda_1$  are   the smallest and the largest among eigenvalues of $A-B$,  and $g(x)=f(|x|)$.
\end{theorem}
\begin{proof}
First assume that $f$ is positive. Then $f$ is   monotone increasing and convex by Lemma \ref{spq-1}. Consider the function     $g(t)= f(|t|)$. Since $g''(t)=f''(|t|)$, the function $g$  is differentiable as many as $f$ is, on any  subset $J$ of $\mathds{R}$  with $0\notin J$. In addition, $g$ is strictly convex if $f$ is so.

Now assume that $A,B\in\mathbb{M}_n^{++}$ are positive definite matrices such that $A-B$ is invertible. Then, the eigenvalues of $A-B$ are contained in a set $J\subseteq\mathds{R}$   with $0\notin J$.   If  $\lambda_n$ and $\lambda_1$  are   the smallest and the largest among eigenvalues of $A-B$, then    Lemma \ref{lm-fmps} implies that
\begin{align}\label{o2}
\begin{split}
\left\langle f( c|A-B|)x,x\right\rangle=  \langle g(Z)x,x\rangle &\leq  \gamma(c\lambda_n, c\lambda_1,g) g(\langle Zx,x\rangle)\\
&=\gamma(c\lambda_n,c\lambda_1,g) f(c|\langle (A-B) x,x\rangle|)  \\
\end{split}
\end{align}
holds for every unit vector $x\in\mathds{C}^n$ and every positive scalar $c\in\mathds{R}^+$.

Next suppose that  $\alpha\in[0,1]$.
It follows from the minimax principle, Lemma \ref{minimax},
that for every $k=1,\cdots,n$, there exists a $k$-dimensional subspace $\mathcal{M}$ of $\mathds{C}^n$ such that
{\small\begin{align}\label{c2}
\begin{split}
\lambda_k^\downarrow\left[f\left((1-\alpha)A+\alpha B\right)\right]&=
\min\left\{\left\langle f\left((1-\alpha)A+\alpha B\right)x,x\right\rangle;\  x\in \mathcal{M}, \|x\|=1 \right\} \\
&=\min\left\{f(\left\langle (1-\alpha)A+\alpha B x,x\right\rangle);\  x\in \mathcal{M}, \|x\|=1 \right\} \\
&\qquad\qquad\qquad\quad(\mbox{by monotonicity of $f$}) \\
&=\min\left\{f((1-\alpha)\langle Ax,x\rangle+\alpha \langle B x,x\rangle);\  x\in \mathcal{M}, \|x\|=1 \right\}.
\end{split}
\end{align}}
Since $f$ is superquadratic and convex, Lemma \ref{spq-2} and \eqref{mp} give
{\small\begin{align}\label{o1}
\begin{split}
 & f((1-\alpha)\langle Ax,x\rangle+\alpha \langle B x,x\rangle)\\
 &\leq (1-\alpha)( f(\langle Ax,x\rangle)-f(\alpha|\langle (A-B) x,x\rangle|))
  +\alpha( f(\langle B x,x\rangle)- f((1-\alpha)|\langle (A-B) x,x\rangle|))\\
&\leq (1-\alpha)( \langle f(A)x,x\rangle-f(\alpha|\langle (A-B) x,x\rangle|))
 +\alpha(  \langle f(B) x,x\rangle- f((1-\alpha)|\langle (A-B) x,x\rangle|)).
\end{split}
\end{align}}

 It follows from \eqref{o1} and  \eqref{o2}   that
 {\small\begin{align*}
 f((1-\alpha)\langle Ax,x\rangle+\alpha \langle B x,x\rangle)\leq \langle T x,x\rangle,
\end{align*}}
for every unit vector $x\in\mathds{C}^n$,  where
{\small\begin{align*}
T &=(1-\alpha)(f(A)-  \gamma(\alpha\lambda_n, \alpha\lambda_1,g)^{-1}  f(\alpha|A-B|) )\\
&\qquad +\alpha (f(B)-  \gamma((1-\alpha)\lambda_n,(1-\alpha)\lambda_1,g)^{-1}  f((1-\alpha)|A-B|) ).
\end{align*}}
 Hence, we have from \eqref{c2} that
{\small\begin{align*}
\lambda_k^\downarrow\left[f\left((1-\alpha)A+\alpha B\right)\right]\leq \min\left\{ \langle T x,x\rangle;\  x\in \mathcal{M}, \|x\|=1 \right\} \leq \lambda_k^\downarrow(T),
\end{align*}}
by using the minimax principle once more. This proves \eqref{thq-1}.

\end{proof}
%%------------------------------------------------------------------------------------
With $\alpha=1/2$ we have the next corollary.
\begin{corollary}
  Let $f:(0,\infty)\to(0,\infty)$ be a twice differentiable superquadratic function  and let $A,B\in\mathbb{M}_n^{++}$. If  $f$  is  strictly convex   and $A-B$ is invertible, then
{\small\begin{align}\label{co1-q3}
 f\left(\frac{A+B}{2}\right)
&\leq U^*\left[\frac{f(A)+f(B)}{2}-\gamma\left(\frac{\lambda_n}{2},  \frac{\lambda_1}{2},g\right)^{-1} f\left(\left|\frac{A-B}{2}\right|\right) \right]U
\end{align}}
for some unitary $U\in\mathbb{M}_n$, where $\lambda_n$ and $\lambda_1$  are   the smallest and the largest among eigenvalues of $A-B$,  respectively.
\end{corollary}
%%------------------------------------------------------------
\begin{corollary}
  With the hypothesis as in Theorem~\ref{th-t1}
  $$\mathrm{tr}\left(f\left(\frac{A+B}{2}\right)\right)\leq
  \mathrm{tr}\left(\frac{f(A)+f(B)}{2}\right)-\mathrm{tr}\left(\gamma\left(\frac{\lambda_n}{2},  \frac{\lambda_1}{2},g\right)^{-1} f\left(\left|\frac{A-B}{2}\right|\right)\right).
  $$
\end{corollary}

%%------------------------------------------------------------------
Typical example of superquadratic functions, which are convex and increasing, are power functions
  $f(x)=x^p$,   when $p\geq2$.    The next result then follows.

%%------------------------------------------------
\begin{corollary}
 Let $A,B\in\mathbb{M}_n^{++}$  such that $A-B$ is invertible.  If $p\geq2$, then
{\small\begin{align}\label{co4-q1}
\lambda_k^\downarrow\left[(A+B)^p\right]
&\leq \lambda_k^\downarrow\left[2^{p-1}(A^p+B^p)-K\left(\frac{\lambda_n}{2}
,\frac{\lambda_1}{2} ,g\right) |A-B|^p \right]\qquad (k=1,\dots,n),
\end{align}}
 where $\lambda_n$ and $\lambda_1$  are   the smallest and the largest among eigenvalues of $A-B$,  respectively and $K(m,M,g)$ is
 $$K(m,M,g)=\frac{m|M|^p-M|m|^p}{(p-1)(M-m)}\left|\frac{p-1}{p}\frac{|M|^p-|m|^p}{m|M|^p-M|m|^p}\right|^p.$$
\end{corollary}
\begin{remark}
  We note that the constant $K(m,M,g)$ is  a more general form of the well-known generalized Kantorovich constant, see \cite[pp. 55]{FMPS}.
\end{remark}

Theorem~\ref{th-t1} gives a refinement for the estimation of eigenvalues in \eqref{original}. We give an example. Let $\alpha=1/2$ and consider the superquadratic function $f(x)=x^3$. Assume that
 \begin{align*}
  A=\begin{pmatrix}
      7 & 3 \\
      3 & 8
    \end{pmatrix} \quad\mbox{and}\qquad  B=\begin{pmatrix}
      1 & 1 \\
      1 & 2
    \end{pmatrix}.
\end{align*}
Then we have
 \begin{align*}
 \lambda^\downarrow\left(\frac{A+B}{2}\right)^3\simeq(282.5,14.5) \quad \lambda^\downarrow\left(\frac{A^3+B^3}{2}\right)\simeq(594.5,44.5)
\end{align*}
and
 \begin{align*}
\lambda^\downarrow \left(\frac{A^3+B^3}{2}-K(2,4 ,g)\left|\frac{A-B}{2}\right|^3 \right)\simeq(505.5,32.7),
\end{align*}
whence
  \eqref{thq-1} implies that
   \begin{align*}
 \lambda^\downarrow\left(\frac{A+B}{2}\right)^3&\leq \lambda^\downarrow \left(\frac{A^3+B^3}{2}-K(2,4 ,g)\left|\frac{A-B}{2}\right|^3 \right)\\
 &\leq
 \lambda^\downarrow\left(\frac{A^3+B^3}{2}\right).
\end{align*}

%----------------------------------------------------------------

\bigskip

  In  our proof of the next result, we use  the trick of passing to $\mathbb{M}_2(\mathbb{M}_n)$  by unitary dealation, as in \cite{FMPS}.

  \begin{proposition}
   Let $f:(0,\infty)\to(0,\infty)$ be a twice differentiable superquadratic function,   $X\in\mathbb{M}_n^{++}$ and let $C\in\mathbb{M}_n $ be an isometry.  If  $f$  is  strictly convex, then there exist  unitaries $U, V \in\mathbb{M}_n $ such that
   \begin{align}\label{dede}
 W^*f(C^*XC)W\leq  C^*f(X)C-\gamma\left(\frac{m}{2},\frac{M}{2},g\right)^{-1}f\left(\left|C^*X\sqrt{I-CC^*}\right|\right)
   \end{align}
in which $W=1/2(U+V)$ and $g(x)=f(|x|)$, provided that
$$mI_{2n}\leq\begin{pmatrix}
          0&  C^*X\sqrt{I-CC^*}    \\
    \sqrt{I-CC^*}XC & 0
 \end{pmatrix}\leq MI_{2n}.$$
  \end{proposition}
 \begin{proof}
  If $C\in\mathbb{M}_n$ is an isometry, then the block matrices
\begin{align*}
  R_1=\begin{pmatrix}
      C & D \\
      0 & -C^*
    \end{pmatrix}\quad\mbox{and}\quad   R_2=\begin{pmatrix}
      C & -D \\
      0 &  C^*
    \end{pmatrix}
\end{align*}
are unitaries in $\mathbb{M}_2(\mathbb{M}_n)$, where $D=(I-CC^*)^{1/2}$. If $X\in\mathbb{M}_n^{++}$, then  considering positive semi-definite matrices $A=R_1^*(X\oplus 0)R_1$ and $B=R_2^*(X\oplus 0)R_2$, we have
\begin{align*}
 \frac{A+B}{2} =
 \begin{pmatrix}
  C^*XC & 0 \\
 0 & DXD
\end{pmatrix}
 \quad\mbox{and}\quad
 \left| \frac{A-B}{2}\right|
 =\begin{pmatrix}
          |C^*XD| & 0 \\
    0 & |DXC|
 \end{pmatrix}.
\end{align*}
Furthermore,
$$\lambda^\downarrow(A-B)=2\lambda^\downarrow\left(\begin{pmatrix}
          0&  C^*XD    \\
    DXC & 0
 \end{pmatrix}\right).$$
 By the hypothesis, $2m\leq \lambda_n(A-B)$ and $\lambda_1(A-B)\leq 2M$.
Employing \eqref{thq-1} with $\alpha=1/2$ and $A,B$ as above entails
{\small\begin{align*}
 &\lambda_k^\downarrow\left(  \begin{pmatrix}
  f(C^*XC) & 0 \\
 0 & f(DXD)
\end{pmatrix}\right)\\
& \leq  \lambda_k^\downarrow\left(
\begin{pmatrix}
  C^*f(X)C-\gamma\left(m,M,g\right)^{-1}f(|C^*XD|) & 0 \\
 0 & Df(X)D-\gamma\left(m,M,g\right)^{-1}f(|DXC|)
\end{pmatrix}\right).
\end{align*}}

Since  $f$ is positive, $f(DXD)$ is a positive definite matrix and so $\lambda^\downarrow(f(C^*XC)\oplus 0)\leq \lambda^\downarrow(f(C^*XC)\oplus f(DXD)$. Hence, there exists a unitary $U\in \mathbb{M}_2(\mathbb{M}_n)$
such that
\begin{align}\label{pdf}
U^*(f(C^*XC)\oplus 0)U\leq Y\oplus Z,
\end{align}
in which we set $Y=C^*f(X)C-\gamma\left(m,M,g\right)^{-1}f(|C^*XD|)$ and $Z=Df(X)D-\gamma\left(m,M,g\right)^{-1}f(|DXC|) $. By \cite[Theorem 2.1]{uni}, there are unitaries $V_1,V_2,V_3,V_4$ in $\mathbb{M}_n$ such that the unitary $U\in \mathbb{M}_2(\mathbb{M}_n)$ can be decomposed as
$$U=\begin{pmatrix}
  V_1 & 0\\
 0 & V_2
\end{pmatrix}\frac{1}{2}
 \begin{pmatrix}
  I+V_3 & I-V_3 \\
I-V_3 & I+V_3
\end{pmatrix}\begin{pmatrix}
  I & 0 \\
 0 & V_4
\end{pmatrix}.
 $$
Accordingly, we have
\begin{align}\label{rep}
&U^*(f(C^*XC)\oplus 0)U\\
&=\frac{1}{4}\begin{pmatrix}
  V_1(I+V_3)f(C^*XC)(I+V_3^*)V_1^* & V_1(I+V_3)f(C^*XC)(I-V_3^*)V_2^* \\
 V_2(I-V_3)f(C^*XC)(I+V_3^*)V_1^* &  V_2(I-V_3)f(C^*XC)(I-V_3^*)V_2^*
\end{pmatrix}\nonumber,
 \end{align}
whence \eqref{pdf} turns into
\begin{align*}
\begin{pmatrix}
Y - M & -N\\
-N^*   & Z-P
\end{pmatrix} \geq0
\end{align*}
in which $M$, $N$ and $P$ are corresponding blocks in \eqref{rep}. Consequently, $M\leq Y$ as required.
\end{proof}

\noindent\textit{Data Availability Statement.} Data sharing not applicable to this article as no datasets were generated or analysed during the current study.

 %------------------------------------------------------------------------------------

\end{document}